\documentclass[12pt]{amsart}%
\usepackage{amsmath}
\usepackage{graphicx}
\usepackage{amscd}
\usepackage{geometry}
\usepackage{amsfonts}
\usepackage{amssymb}
\usepackage[mathscr]{eucal}%
\theoremstyle{plain}
\newtheorem*{theorem}{Theorem}
\newtheorem*{Corollary 1}{Corollary 1}
\newtheorem*{Corollary 2}{Corollary 2}
\newtheorem*{Theorem 2 of [1]}{Theorem 2 of [1]}

\newtheorem*{definition}{Definition}
\newtheorem{lemma}{Lemma}

\numberwithin{equation}{section}

\begin{document}

% after \begin{document}
\baselineskip=17pt

\title[ ]{On homogeneous locally conical spaces}

\author{Fredric D. Ancel}

\address{Department of Mathematical Sciences, University of Wisconsin - Milwaukee, Box 413, Milwaukee, WI 53201-0413}

\email{ancel@uwm.edu}

\author{David P. Bellamy}

\address{Department of Mathematics, University of Delaware, Newark, DE 19716}

\email{bellamy@udel.edu}

\thanks{}

\date{\today}
\subjclass[2010]{Primary 54B15, 54F15, 54H99; Secondary 54H15, 57N15, 57S05}
\keywords{}

\begin{abstract}
The main result of this article is:
\begin{theorem}
Every homogeneous locally conical connected separable metric space that is not a $1$-manifold is strongly $n$-homogeneous for each $n \geq 2$.  Furthermore, every homogeneous locally conical separable metric space is countable dense homogeneous.
\end{theorem}
\noindent This theorem has the following two consequences.
\begin{Corollary 1} 
If $X$ is a homogeneous compact suspension, then $X$ is an absolute suspension (i.e., for any two distinct points $p$ and $q$ of $X$, there is a homeomorphism from $X$ to a suspension that maps $p$ and $q$ to the suspension points).
\end{Corollary 1}
\begin{Corollary 2}
If there exists a locally conical counterexample $X$ to the Bing-Borsuk Conjecture (i.e., $X$ is a locally conical homogeneous Euclidean neighborhood retract that is not a manifold), then each component of $X$ is strongly $n$-homogeneous for all $n \geq 2$ and $X$ is countable dense homogeneous.
\end{Corollary 2}
\end{abstract}

\maketitle

\section{Introduction}

We first provide definitions of the terms used in the abstract.

\begin{definition}
A space is \emph{locally conical} if every point of the space has an open neighborhood that is homeomorphic to the open cone over a compact space.
\end{definition}

\begin{definition}
A space $X$ is \emph{homogeneous} if for any two points $p$ and $q$ of $X$, there is a homeomorphism of $X$ that maps $p$ to $q$.  More generally, for $n \geq 1$, a space $X$ is \emph{strongly $n$-homogeneous} if every bijection between two $n$-element subsets of $X$ can be extended to a homeomorphism of $X$.  Thus, a space is homogeneous if and only if it is strongly $1$-homogeneous.
\end{definition}

\begin{definition}
A space $X$ is \emph{countable dense homogeneous} if for any two countable dense subsets $A$ and $B$ of $X$, there is a homeomorphism of $X$ that maps $A$ onto $B$.
\end{definition}

\indent The Theorem (stated in the abstract) asserts that every homogeneous locally conical connected separable metric space that is not a $1$-manifold is strongly $n$-homogeneous for each $n \geq 2$.  The hypothesis of this theorem must exclude $1$-manifolds because the connected $1$-manifolds, $\mathbb{R}$ and $S^1$, while homogeneous, fail to be strongly $3$-homogeneous and strongly $4$-homogeneous, respectively.

\indent The pseudo-arc as well as the space $S^1 \times \mu^1$, where $\mu^1$ is the $1$-dimensional Menger universal curve, provide evidence that the ``locally conical'' hypothesis or something similar is necessary.  Neither of these spaces is locally conical, both are homogeneous, but neither is strongly $2$-homogeneous.  The pseudo-arc is not strongly $2$-homogeneous because it has uncountably many pairwise disjoint composants (\cite{18}, Theorems 11.15 and 11.17, pages 203-204) and two points in the same composant can't be mapped by a homeomorphism to two points in different composants.   See \cite{13} for a proof that $S^1 \times \mu^1$ is not strongly $2$-homogeneous.  Apparently, for $n \geq 4$, it is not known whether there exists a space that is strongly $n$-homogeneous but not strongly $(n+1)$-homogeneous.

\begin{definition}
A compact space $X$ is an \emph{absolute suspension} if for any two distinct points $p$ and $q$ of $X$, there is a homeomorphism from $X$ to a suspension that maps $p$ and $q$ to the suspension points.
\end{definition}

\indent The concept of an absolute suspension originated in the paper \cite{11}.  In that paper, de Groot conjectured that for $n \geq 1$, every $n$-dimensional absolute suspension is homeomorphic to the $n$-sphere.  This conjecture is known to be true for $n \leq 3$ (\cite{20} and \cite{17}) and is currently unresolved for $n \geq 4$.  See \cite{2} for further information about absolute suspensions.  Since the suspension points of a suspension can be interchanged by an obvious homeomorphism, then every absolute suspension is clearly homogeneous.  It is natural to ask whether the converse of this statement is true: is every homogeneous suspension an absolute suspension?  According to Corollary 1, the answer is ``yes''.  Observe that proving Corollary 1 is equivalent to proving that every homogeneous suspension is strongly $2$-homogeneous.  Notice that the Theorem actually implies more: it implies that every homogeneous suspension is strongly $n$-homogeneous for each $n \geq 2$ and countable dense homogeneous.

\begin{definition}
A space is a \emph{Euclidean neighborhood retract} if it is homeomorphic to a retract of an open subset of $\mathbb{R}^n$ for some $n \geq 1$.
\end{definition}

\indent The \textit{Bing-Borsuk Conjecture} originated in the paper \cite{6} and asserts that every homogeneous Euclidean neighborhood retract is a topological manifold without boundary.  It is conceivable that the methods of the authors of \cite{8} might yield a locally conical homogeneous Euclidean neighborhood retract that is a homology manifold but not a manifold.  If such a space exists, then Corollary $2$ applies to it.

\indent It is clear that Corollaries $1$ and $2$ follow from the Theorem.

\section{Locally conical spaces}

\indent A locally conical space is one that is covered by open cone neighborhoods.  We now elaborate on this concept and prove three lemmas that are the necessary ingredients for our proof of the Theorem.

\begin{definition}
Let $U$ be an open subset of a metric space $X$.  A \emph{cone chart for $U$} is a proper map $\phi : Y \times (0,\infty] \rightarrow U$ such that 
\begin{itemize}
	\item $Y$ is a compact metric space,
	\item $\phi(Y \times \{ \infty \}) = \{p\})$ for some point $p \in U$ and
	\item $\phi$ maps $Y \times (0,\infty)$ homeomorphically onto $U - \{p\}$.
\end{itemize}
\end{definition}
\noindent Equivalently, $\phi : Y \times (0,\infty] \rightarrow U$ is a cone chart for $U$ if $Y$ is a compact metric space and $\phi$ induces a homeomorphism from the quotient space $Y \times (0,\infty]/Y \times \{ \infty \}$ onto $U$.  In this situation, we call $p$ the \emph{vertex} of $\phi$ and we call $U$ an \emph{open cone neighborhood}.  (The properness of the map $\phi : Y \times (0,\infty] \rightarrow U$ makes $\phi$ a closed map.  This prevents a sequence in $Y \times (0,\infty]$ that converges to $Y \times \{0\}$ from having an image in $U$ that converges to $p$.)

\indent The following notion is central to the subsequent proofs.

\begin{definition}
Let $\phi : Y \times (0,\infty] \rightarrow U$ and $\psi : Z \times (0,\infty] \rightarrow V$ be cone charts for open subsets $U$ and $V$ of a metric space $X$.  $\phi$ and $\psi$ are \emph{$2$-interlaced} if
\begin{itemize}
	\item $\phi(Y \times (1,\infty]) \supset \psi(Z \times [2,\infty])$,
	\item $\psi(Z \times (2,\infty]) \supset \phi(Y \times [3,\infty])$ and
	\item $\phi(Y \times (3,\infty]) \supset \psi(Z \times [4,\infty])$.
\end{itemize}
\end{definition}
\noindent (See Figure $1$.)  Observe that if $\phi$ and $\psi$ are $2$-interlaced, then the vertices of $\phi$ and $\psi$ both lie in $U \cup V$.

\begin{figure}[h]
\includegraphics{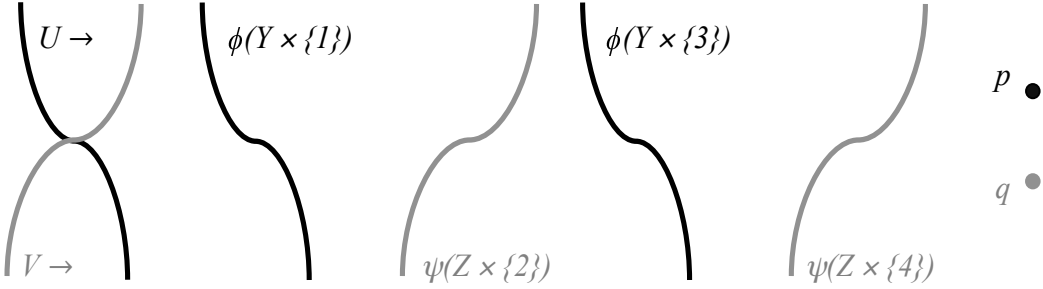}
\caption{$\phi$ and $\psi$ are $2$-interlaced cone charts}
\end{figure}

\indent Before formally stating the three lemmas, we paraphrase them and sketch how they lead to a proof of the Theorem.  Lemma $1$ says that, given two $2$-interlaced cone charts, then there is a homeomorphism supported on the intersection of the images of the two cone charts that moves the vertex of one cone chart to the vertex of the other.  Lemma $2$ says that if $U$ is an open cone neighborhood in a homogeneous space $X$, then for any two points $x$ and $y$ of $U$, there is a homeomorphism of $X$ supported on $U$, that moves $x$ to $y$.  Lemma $2$ follows from Lemma $1$ and a theorem of E. G. Effros \cite{12}.  Lemma $3$ says that a homogeneous locally conical connected space that is not a $1$-manifold is not separated by any finite subset.  To prove the strong $n$-homogeneity conclusion of the Theorem, we proceed by induction on $n$.  Assume $X$ is a strongly $n$-homogeneous locally conical connected separable metric space that is not a $1$-manifold.  Given two $(n+1)$-element subsets $\{ p_1, \dots , p_{n+1} \}$ and $\{ q_1, \dots , q_{n+1} \}$ of $X$, the inductive hypothesis allows us to assume $p_i = q_i$ for $1 \leq i \leq n$.  Use Lemma $3$ to join $p_{n+1}$ to $q_{n+1}$ by a path in $X$ that avoids $\{ q_1, \dots , q_n \}$.  Cover this path by a finite chain $U_1, \dots , U_k$ of open cone neighborhoods, and use Lemma $2$ to obtain a sequence of homeomorphisms $h_1, \dots , h_k$ of $X$ such that each $h_i$ is supported on $U_i$, and the composition $h_k \circ \dots \circ h_1$ moves $p_{n+1}$ through the $U_i$'s to $q_{n+1}$.  This proves $X$ is strongly $(n+1)$-homogeneous.  To prove $X$ is countable dense homogeneous, consider two countable dense subsets $A$ and $B$ of $X$.  The rough outline of this argument is to use Lemma $2$ to obtain a sequence $g_1, g_2, \dots$ of homeomorphisms of $X$ that are supported on progressively smaller open cone neighborhoods so that the finite compositions $g_k \circ \dots \circ g_1$ move progressively larger finite subsets of $A$ into $B$ while the inverses $(g_1 \circ \dots \circ g_k)^{-1}$ move progressively larger finite subsets of $B$ into $A$.  Care must be taken to insure that the compositions $g_k \circ \dots \circ g_1$ converge to a homeomorphism of $X$.  Specifically, the supports of the $g_i$'s must be chosen so small that the sequence of compositions $g_k \circ \dots \circ g_1$ forms a Cauchy sequence with respect to a complete metric on the homeomorphism group of $X$.

\indent The fundamental property of $2$-interlaced cone charts is expressed by the following lemma.

\begin{lemma}
If $\phi : Y \times (0,\infty] \rightarrow U$ and $\psi : Z \times (0,\infty] \rightarrow V$ are $2$-interlaced cone charts for open subsets $U$ and $V$ of a metric space $X$ with vertices $p$ and $q$, respectively, then there is a homeomorphism of $X$ that maps $p$ to $q$ and is supported on $U \cup V$.\footnote{We know two different proofs of this lemma.  The one we give here is related in spirit to the proofs in \cite{16}, \cite{10}, \cite{19} and \cite{9}. Those proofs employ a logical gambit known as the \textit{Eilenberg swindle} or the \textit{Eilenberg-Mazur swindle} which manipulates even and odd terms in apparently different infinite unions, exploiting associative and commutative properties, to prove the unions are homeomorphic.  Our proof uses similar ideas, although the swindle is not explicitly visible.  The second proof of this lemma is closely related to ideas in \cite{7}.  (Also see \cite{14} and \cite{15}.)  We outline it in an appendix.}
\end{lemma}

\begin{proof}
Let $A_0 = \phi(Y \times [1,\infty]) - \psi(Z \times (2,\infty])$, $B_1 = C_1 = \psi(Z \times [2,\infty]) - \phi(Y \times (3,\infty])$ and $D_1 = \phi(Y \times [3,\infty]) - \psi(Z \times (4,\infty])$.  (See Figure 2.)

\begin{figure}[h]
\includegraphics{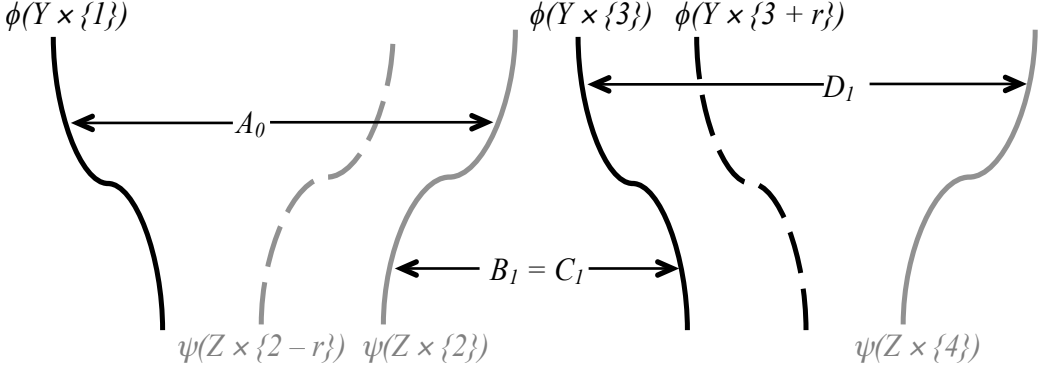}
\caption{$A_0$, $B_1$, $C_1$ and $D_1$}
\end{figure}

\indent Define \textit{shift maps} $S : U \rightarrow U$ and $T : V \rightarrow V$ by $S(\phi(y,t)) = \phi(y,t+2)$ for $(y,t) \in Y \times (0,\infty)$, $S(p) = p$, $T(\psi(z,t)) = \psi(z,t+2)$ for $(z,t) \in Z \times (0,\infty)$ and $T(q) = q$. For $n \geq 1$, define $A_n = S^{(n)}(A_0)$, $B_n = S^{(n-1)}(B_1)$, $C_n = T^{(n-1)}(C_1)$ and $D_n = T^{(n-1)}(D_1)$.  (Here, if $f$ is a function and $n \geq 1$, let $f^{(n)}$ denote the $n$-fold composition of $f$ and let $f^{(0)} = id$.)  Observe that 
\begin{align*}
		&U  =  \phi(Y \times (0,1]) \cup A_0 \cup (\bigcup_{n \geq 1}(B_n \cup A_n)) \cup \{p\},\\
		&V  =  \psi(Y \times (0,2]) \cup (\bigcup_{n \geq 1} (C_n \cup D_n)) \cup \{q\}
\intertext{and}
		&(X - U)\cup\phi(Y \times (0,1] \cup A_0 = (X - V)\cup\psi(Y \times (0,2]).
\end{align*}
\noindent We will construct a homeomorphism $h : X \rightarrow X$ which is the identity on $(X - U)\cup\phi(Y \times (0,1])\cup A_0$ and for $n \geq 1$ maps $B_n$ to $C_n$ and $A_n$ to $D_n$, and maps $p$ to $q$.

\indent The first step in the construction of $h$ is to construct a homeomorphism $\alpha : A_0 \rightarrow D_1$ such that $\alpha = S$ on $\phi(Y \times \{1\})$ and $\alpha = T$ on $\psi(Z \times \{2\})$. $\alpha$ is defined to be the composition of the homeomorphisms $\beta : A_0 \rightarrow A_0 \cup B_1 \cup D_1$ and $\gamma : A_0 \cup B_1 \cup D_1 \rightarrow D_1$ satisfying $\beta = id$ on $\phi(Y \times \{1\})$, $\beta = T$ on $\psi(Z \times \{2\})$, $\gamma = S$ on $\phi(Y \times \{1\})$ and $\gamma = id$ on $\psi(Z \times \{4\})$.  To obtain $\beta$ and $\gamma$, first choose $r > 0$ so that $\phi(Y \times (1,\infty]) \supset \psi(Z \times [2 - r,\infty])$ and $\phi(Y \times (3 + r,\infty]) \supset \psi(Z \times [4,\infty])$. Let $\lambda : [2 - r,2] \rightarrow [2 - r,4]$ and $\mu : [1,3 + r] \rightarrow [3,3 + r]$ be orientation preserving homeomorphisms.  Let $\beta = id$ on $\phi(Y \times [1,\infty]) - \psi(Z \times (2 - r,\infty])$ and $\beta(\psi(z,t)) = \psi(z,\lambda(t))$ for $(z,t) \in Z \times [2 - r,2]$. Thus, $\beta = T$ on $\psi(Z \times \{2\})$. Let $\gamma(\phi(y,t)) = \phi(y,\mu(t))$ for $(y,t) \in Y \times [1,3 + r]$ and $\gamma = id$ on $\phi(Y \times [3 + r,\infty]) - \psi(Z \times (4,\infty])$. Thus, $\gamma = S$ on $\phi(Y \times \{1\})$. Finally let $\alpha = \gamma\circ\beta$.  (See Figures 3\textsc{a} and 3\textsc{b}.)

\renewcommand\thefigure{3a}
\begin{figure}[h]
\includegraphics{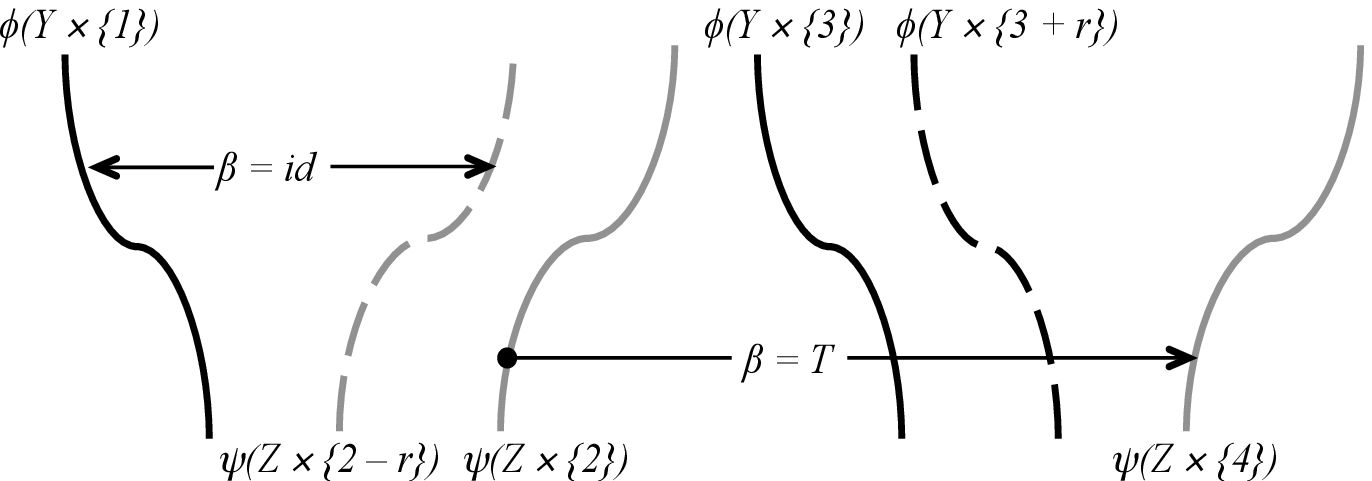}
\caption{$\beta$}
\end{figure}

\renewcommand\thefigure{3b}
\begin{figure}[h]
\includegraphics{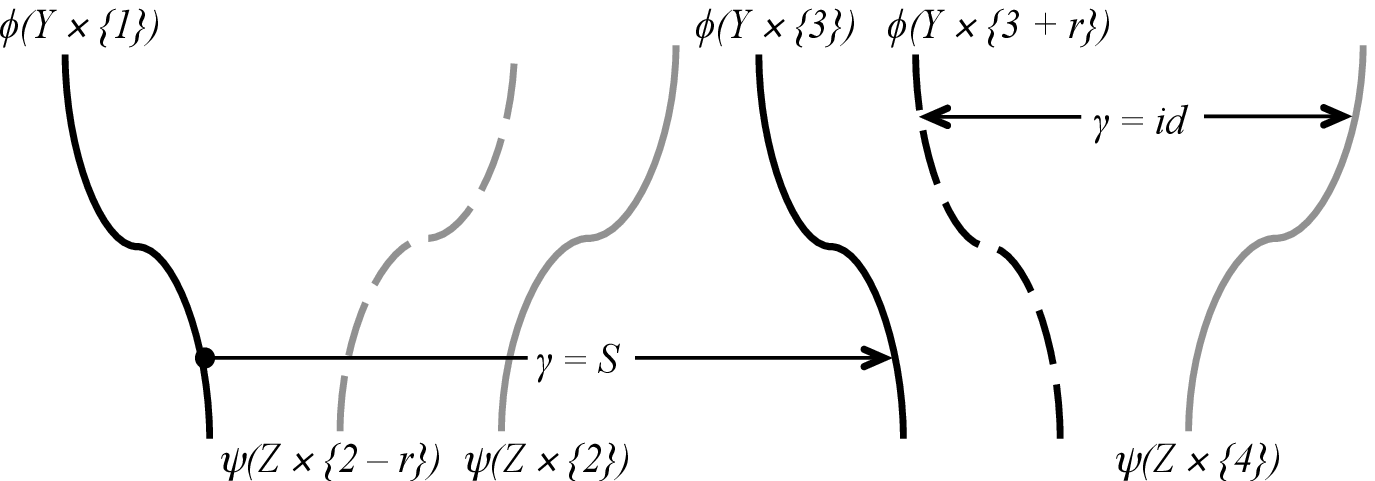}
\caption{$\gamma$}
\end{figure}

\indent The homeomorphism $h$ will be constructed as the union of constituent homeomorphisms between ``pieces'' of $X$.  The constituent homeomorphisms are defined as follows.  $h_0 = id$ on $(X - U) \cup \phi(Y \times (0,1]) \cup A_0 = (X - V) \cup \psi(Z \times (0,2])$. For $n \geq 1$, the homeomorphism $h_n^B : B_n \rightarrow C_n$ is defined by $h_n^B = T^{(n - 1)}\circ(S^{(n - 1)})^{-1} | B_n$ and the homeomorphism $h_n^A : A_n \rightarrow D_n$ is defined by $h_n^A = T^{(n - 1)}\circ\alpha\circ(S^{(n)})^{-1} | A_n$.  (Thus, $h_1^B = id$ on $B_1$ and $h_1^A = \alpha\circ S^{-1}$ on $A_1$). Finally $h_{\infty} : \{p\} \rightarrow \{q\}$.  Now define $h : X \rightarrow X$ as
\begin{center}
$h = h_0 \cup ( h_1^B \cup h_1^A \cup h_2^B \cup h_2^A \cup \dots ) \cup h_\infty$.
\end{center} 
(See Figure 4.)

\renewcommand\thefigure{4}
\begin{figure}[h]
\includegraphics{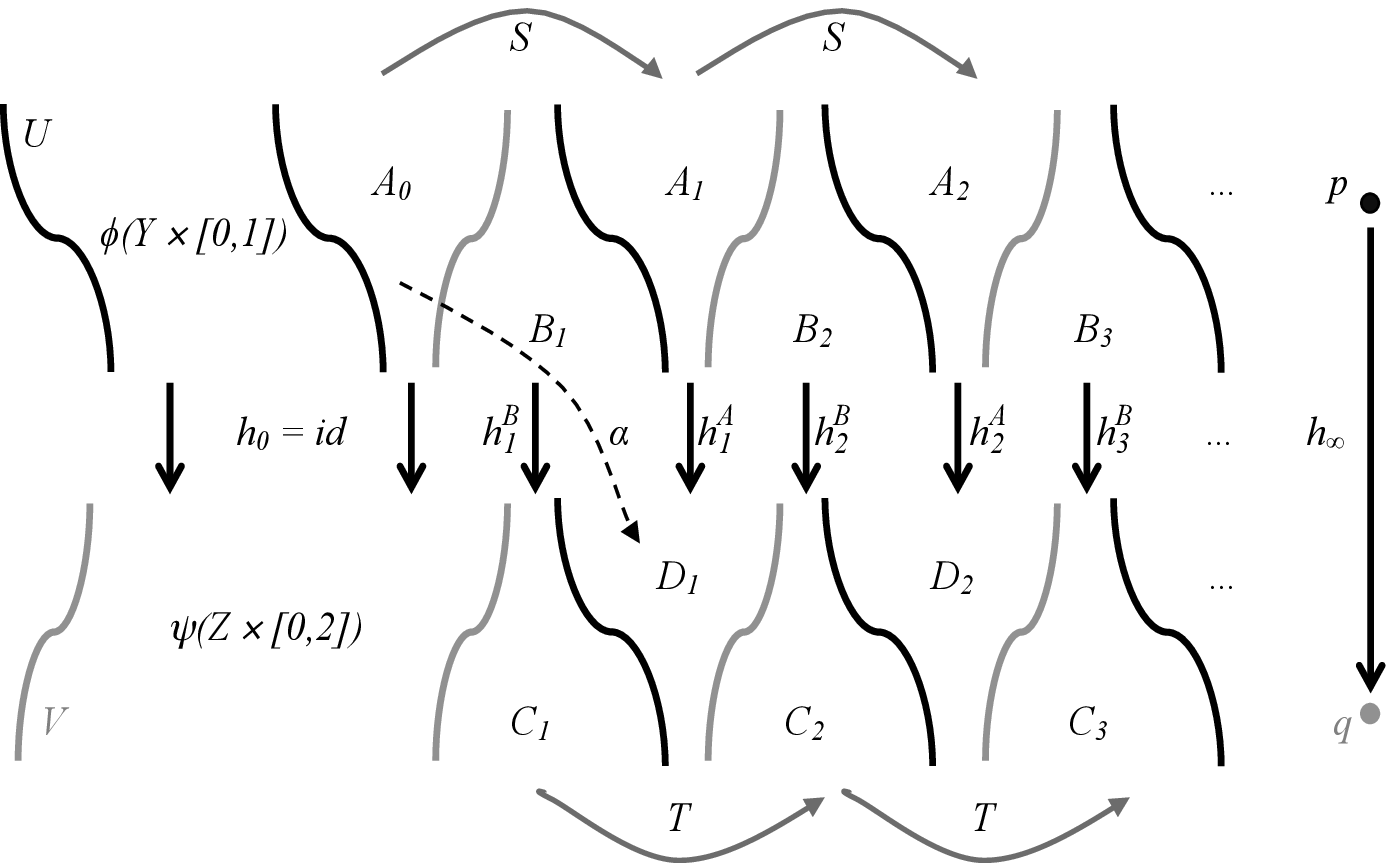}
\caption{$h = h_0 \cup ( h_1^B \cup h_1^A \cup h_2^B \cup h_2^A \cup \dots ) \cup h_\infty$}
\end{figure}

\indent To prove that $h$ is well-defined, we must show that its constituents agree wherever their domains overlap.  To this end, observe that since $h_0 = id$ and $h_1^B = id$, then these functions agree on $((X - U) \cup \phi(Y \times (0,1]) \cup A_0) \cap B_1 = \psi(Z \times \{2\})$.  Let $n \geq 1$. Note that $B_n \cap A_n = S^{(n - 1)}(B_1 \cap S(A_0)) =  S^{(n - 1)}(\phi(Y \times \{3\})) = S^{(n)}(\phi(Y \times \{1\}))$.  Hence, 
\begin{align*}
	&h_n^B | B_n \cap A_n = T^{(n - 1)}\circ(S^{(n - 1)})^{-1} | S^{(n)}(\phi(Y \times \{1\}))
\intertext{and}
	&h_n^A | B_n \cap A_n = T^{(n - 1)}\circ\alpha\circ(S^{(n)})^{-1} | S^{(n)}(\phi(Y \times \{1\})) =\\
	&T^{(n - 1)}\circ(\alpha | \phi(Y \times \{1\}))\circ(S^{(n)})^{-1} | S^{(n)}(\phi(Y \times \{1\})) =\\
	&T^{(n - 1)}\circ S\circ(S^{(n)})^{-1} | S^{(n)}(\phi(Y \times \{1\})) =\\
	&T^{(n - 1)}\circ(S^{(n - 1)})^{-1} | S^{(n)}(\phi(Y \times \{1\}))
\end{align*}
Thus, $h_n^B =h^A_n$ on $B_n \cap A_n$. Also note that $A_n \cap B_{n + 1} = S^{(n)}(A_0 \cap B_1) =  S^{(n)}(\psi(Z \times \{2\}))$. Hence,
\begin{align*}
	&h_n^A | A_n \cap B_{n + 1} = T^{(n - 1)}\circ\alpha\circ(S^{(n)})^{-1} | S^{(n)}(\psi(Z \times \{2\})) =\\
	&T^{(n - 1)}\circ(\alpha | \psi(Z \times \{2\}))\circ(S^{(n)})^{-1} | S^{(n)}(\psi(Z \times \{2\})) =\\
	&T^{(n - 1)}\circ T\circ(S^{(n)})^{-1} | S^{(n)}(\psi(Z \times \{2\})) =\\
	&T^{(n)}\circ(S^{(n)})^{-1} | S^{(n)}(\psi(Z \times \{2\}))
\intertext{and}
	&h_{n + 1}^B | A_n \cap B_{n + 1} = T^{(n)}\circ(S^{(n)})^{-1} | S^{(n)}(\psi(Z \times \{2\}))
\end{align*}
Thus, $h_n^A =h^B_{n + 1}$ on $A_n \cap B_{n + 1}$.  We conclude that $h$ is well defined.

\indent Since $\{ (X - U) \cup \phi(Y \times (0,1]) \cup A_0 \} \cup \{ B_n : n \geq 1 \} \cup \{ A_n : n \geq 1 \}$ is a locally finite closed cover of $X - \{p\}$ and the restriction of $h$ to each element of this cover is continuous, then $h$ is continuous at every point of $X - \{p\}$.  The collections $\{ (\bigcup_{n \geq k}(B_n \cup A_n)) \cup \{p\} : k \geq 1 \}$ and $\{ (\bigcup_{n \geq k}(C_n \cup D_n)) \cup \{q\} : k \geq 1 \}$ are bases for the topology on $X$ at $p$ and $q$, respectively, and $h$ maps the elements of the first collection to the elements of the second collection.  Hence, $h$ is continuous at $p$.

\indent $h$ is a bijection and the constituents of $h$ are themselves homeomorphisms.  
Hence, $h^{-1}$ exists and is given by the formula
\begin{center}
$h^{-1} = (h_0)^{-1} \cup ( (h_1^B)^{-1} \cup (h_1^A)^{-1} \cup (h_2^B)^{-1} \cup (h_2^A)^{-1}  \cup \dots ) \cup (h_\infty)^{-1}$.
\end{center}
Reasoning as before, we see that $h^{-1}$ is well-defined and continuous.  Thus, $h : X \rightarrow X$ is a homeomorphism.  Clearly, $h$ maps $p$ to $q$ and is supported on $U \cap V$. 
\end{proof}

\begin{lemma}
If $U$ is an open cone neighborhood in a homogeneous separable metric space $X$, and $x$ and $y \in U$, then there is a homeomorphism of $X$ that maps $x$ to $y$ and is supported on $U$.
\end{lemma}

\indent The proof of Lemma 2 relies on Lemma 1 and a consequence of a theorem of E. G. Effros \cite{12}. The precise statement of the version of Effros' theorem that we need here appears as Theorem 2 of \cite{1}.  We quote it:

\begin{Theorem 2 of [1]}
Suppose $X$ is a locally compact separable metric space and $\mathscr{H}(X)$ (the homeomorphism group of $X$) is endowed with the complemented compact-open topology.  If $X$ is homogeneous, then $X$ is micro-homogeneous.
\end{Theorem 2 of [1]}

\indent We now explain the less familiar terms in this assertion.  For subsets $A$ and $B$ of a space $X$, let $<$$A,B$$>$ $= \{ h \in \mathscr{H}(X) : h(A) \subset B \}$.  The \emph{complemented compact-open topology} on $\mathscr{H}(X)$ is the topology determined by the basis consisting of all finite intersections of sets that are either of the form $<$$C,U$$>$ or $<$$X - U,X - C$$>$  where $C$ is a compact subset of $X$ and $U$ is an open subset of $X$.  The complemented compact-open topology on $\mathscr{H}(X)$ is metrizable.  Moreover, if $\rho$ is the restriction to $X$ of a metric on the one-point compactification of $X$, then the supremum metric on $\mathscr{H}(X)$ determined by $\rho$ induces the complemented compact-open topology on $\mathscr{H}(X)$.  Saying that $X$ is \emph{micro-homogeneous} (with respect to the complemented compact-open topology on $\mathscr{H}(X)$) means that for every open subset $\mathscr{U}$ of $\mathscr{H}(X)$ and every $x \in X$, $\{ h(x) : h \in \mathscr{U} \}$ is an open subset of $X$.  (In other words, $X$ is micro-homogeneous if for each $x \in X$, the function $h \mapsto h(x) : \mathscr{H}(X) \rightarrow X$ is an open map.)  See section 5 of \cite{1} for details.

\begin{proof}[Proof of Lemma 2.]
Assume $X$ is a homogeneous separable metric space and $U$ is an open cone neighborhood in $X$.  Since $U$ is locally compact and $X$ is homogeneous, then $X$ is locally compact.  Thus, Theorem 2 of \cite{1} implies that $X$ is micro-homogeneous.

\indent Let $\phi : Y \times (0,\infty] \rightarrow U$ be a cone chart for $U$ with vertex $p$. It suffices to prove that every point of $U$ can be mapped to $p$ by a homeomorphism of $X$ that is supported on $U$.

\indent Let $\rho$ be the restriction to $X$ of a metric on the one-point compactification of $X$, and let $\sigma$ be the supremum metric on $\mathscr{H}(X)$ determined by $\rho$.  Let $\epsilon > 0$ be chosen so that $\epsilon$ is less than the distances (with respect to $\rho$) between the following three pairs of sets:

\begin{itemize}
	\item $\phi(Y \times [2,\infty])$ and $X - \phi(Y \times (1,\infty])$,
	\item $\phi(Y \times [3,\infty])$ and $X - \phi(Y \times (2,\infty])$ and
	\item $\phi(Y \times [4,\infty])$ and $X - \phi(Y \times (3,\infty])$.
\end{itemize}

\noindent ($\epsilon$ exists because these pairs of sets are disjoint and closed and the first set in each pair is compact.)  The virtue of this choice of $\epsilon$ is that if $h \in \mathscr{H}(X)$ and $\sigma(h,id_X) < \epsilon$, then $h\circ\phi : Y \times (0,\infty] \rightarrow h(U)$ is a cone chart for $h(U)$ such that $\phi$ and $h\circ\phi$ are $2$-interlaced.

\indent Let $\mathscr{U} = \{ h \in \mathscr{H}(X) : \sigma(h,id_X) < \epsilon \}$.  Then $\mathscr{U}$ is an open neighborhood of $id_X$ in $\mathscr{H}(X)$.  Since $X$ is micro-transitive, then $N = \{ h(p) : h \in \mathscr{U} \}$ is an open neighborhood of $p$ in $X$.

\indent Let $x \in U$.  Using the cone structure on $U$ induced by $\phi$, construct a homeomorphism $g : X \rightarrow X$ supported on $U$ that slides $x$ toward $p$ so that $g(x) \in N$.  Hence, there is an $h \in \mathscr{U}$ such that $h(p) = g(x)$.  Therefore, $\sigma(h,id_X) < \epsilon$.  Thus, $h\circ\phi : Y \times (0,\infty] \rightarrow h(U)$ is a cone chart for $h(U)$ with vertex $g(x)$ such that $\phi$ and $h\circ\phi$  are $2$-interlaced.  Lemma 1 now implies there is a homeomorphism $f : X \rightarrow X$ such that $f(p) = g(x)$ and $f$ is supported on $U \cap h(U)$.  Hence, $f^{-1}\circ g : X \rightarrow X$ is a homeomorphism that maps $x$ to $p$ and is supported on $U$.
\end{proof}

\begin{lemma}
If $X$ is a homogeneous locally conical connected separable metric space that is not a $1$-manifold, and $F$ is a finite subset of $X$, then $X - F$ is path connected.
\end{lemma}

\begin{proof}
Since open cone neighborhoods are locally compact and connected, then $X$ is a locally compact, locally connected connected metric space.  Hence, $X$ is path connected.  Suppose $x_0$ and $x_1$ are points in $X - F$ and $\alpha : [0,1] \rightarrow X$ is a path joining $x_0$ to $x_1$.  We will show how to modify $\alpha$ so that it misses $F$ without changing its endpoints.  Suppose $p \in \alpha([0,1]) \cap F$.  Let $U$ be an open cone neighborhood of $p$ that is disjoint from $(F - \{p\}) \cup \{x_0,x_1\}$ and let $\phi : Y \times (0,\infty] \rightarrow U$ be a cone chart for $U$.  Lemma 2 allows us to assume that $p$ is the vertex of $\phi$.  We now describe a process for modifying $\alpha$ so that its image misses $p$ without introducing any new intersections of $\alpha([0,1])$ with $F$ and without  changing $\alpha$ near its endpoints.  Repeating this process for each point of $\alpha([0,1]) \cap F$ will move $\alpha$ to a path joining $x_0$ to $x_1$ in $X - F$.

\indent Let $V = \phi(Y \times (1,\infty])$. Then $V$ is an open cone neighborhood because a cone chart for $V$ can be defined by composing $\phi | Y \times (1,\infty]$ on the right with an obvious homeomorphism from $Y \times (0,\infty]$ to $Y \times (1,\infty]$. Let $t_0 = min\  \alpha^{-1}(\phi(Y \times [1,\infty]))$ and $t_1 = max\  \alpha^{-1}(\phi(Y \times [1,\infty]))$.  Let $y_0$ and $y_1 \in Y$ so that $\alpha(t_0) = \phi(y_0,1)$ and $\alpha(t_1) = \phi(y_1,1)$.  Modify $\alpha$ so that $\alpha([t_0,t_1]) = \phi(\{y_0,y_1\} \times [1,\infty])$ without moving $\alpha | [0,t_0] \cup [t_1,1]$.  Now $\alpha([0,1]) \cap V = \phi(\{y_0,y_1\} \times (1,\infty])$.  (Thus, $\alpha([0,1]) \cap V$ is an open cone over two points with vertex $p$.)  It is still the case that $p = \phi(Y \times \{\infty\}) \in \alpha([0,1]) \cap V$.

\indent Next we observe that $Y$ must have at least three elements.  Indeed, if $Y$ has only one element, say $Y = \{z_1\}$, then $U = \phi(\{z_1\} \times (0,\infty])$ is an open subset of $X$ that is homeomorphic to $(0,\infty]$.  This is impossible because $X$ is homogeneous.  If $Y$ has two elements, say $Y = \{z_1,z_2\}$, then $U = \phi(\{z_1,z_2\} \times (0,\infty])$ is an open subset of $X$ that is homeomorphic to $\mathbb{R}$.  Since $X$ is homogeneous, then it follows that $X$ must be a $1$-manifold which is ruled out by hypothesis.  

\indent Since $Y$ has at least three elements, then there is a point $y_2 \in Y - \{y_0,y_1\}$.  Let $q = \phi(y_2,2)$.  Then $q \in V - \alpha([0,1])$.  Lemma 2 provides a homeomorphism $h : X \rightarrow X$ supported on $V$ such that $h(q) = p$.  Then $h\circ\alpha : [0,1] \rightarrow X$ is a path joining $x_0$ to $x_1$ whose image misses $p$.  Repeating this process for each point of $\alpha([0,1]) \cap F$ will move $\alpha$ to a path joining $x_0$ to $x_1$ in $X - F$.
\end{proof}

\section{The proof of the Theorem}

\begin{proof}[\nopunct]
\ \ First suppose $X$ is a homogeneous locally conical connected separable metric space that is not a $1$-manifold.  We will prove that $X$ is strongly $n$-homogeneous for each $n \geq 2$ by induction. Let $n \geq 1$ and assume $X$ is strongly $n$-homogeneous.  Suppose $A = \{ p_1, \dots , p_n, p_{n + 1} \}$ and $B = \{ q_1, \dots , q_n, q_{n + 1} \}$ are two $(n + 1)$-element subsets of $X$.  By inductive hypothesis, there is a homeomorphism $f : X \rightarrow X$ such that $f(p_i) = q_i$ for $1 \leq i \leq n$.  Let $F = \{ q_1, \dots , q_n \}$.   Lemma 3 provides an arc $A$ joining $f(p_{n + 1})$ to $q_{n + 1}$ in $X - F$.  We can cover the arc $A$ with a finite sequence of open cone neighborhoods $U_1, \dots , U_k$ such that $f(p_{n + 1}) \in U_1$, $q_{n + 1} \in U_k$, $U_i \cap U_{i + 1} \neq \varnothing$ for $1 \leq i < k$, and $U_1 \cup \dots \cup U_k \subset X - F$.  Let $x_0 = f(p_{n + 1})$ and $x_k = q_{n + 1}$, and for $1 \leq i < k$, choose $x_i \in U_i \cap U_{i + 1}$.  For $1 \leq i \leq k$, Lemma 2 provides a homeomorphism $g_i : X \rightarrow X$ supported on $U_i$ such that $g_i(x_{i - 1}) = x_i$.  Thus, $g_k \circ \dots \circ g_1 : X \rightarrow X$ is a homeomorphism supported on $U_1 \cup \dots \cup U_k$ such that $g_k \circ \dots \circ g_1 \circ f(p_{n + 1}) = q_{n + 1}$.  Since $g_k \circ \dots \circ g_1 = id$ on $F$, then for $1 \leq i \leq n$, $g_k \circ \dots \circ g_1 \circ f(p_i) = g_k \circ \dots \circ g_1(q_i) = q_i$.  It follows that $X$ is strongly $(n + 1)$-homogeneous.

\indent To prove the second sentence of the Theorem, we invoke Theorem 3 of \cite{4}.  Assume $X$ is a homogeneous locally conical separable metric space.  (Here we allow the possibility that $X$ is not connected, and we don't exclude the possibility that it is a $1$-manifold.)  Lemma 2 of this paper implies that $X$ is, in the terminology of \cite{4}, \emph{strongly locally homogeneous}.  Hence, Theorem 3 of \cite{4} implies that $X$ is countable dense homogeneous. 
\end{proof}

\section{Question}

\indent The referee has asked whether \textit{the results of this paper can be generalized to spaces that are non locally compact.  Specifically, can they be generalized to separable complete metric spaces?}  Some aspects of the argument can be generalized to the separable complete metric setting.  The definition of \textit{cone chart} can be modified so that all the lemmas except Lemma 2 hold in the more general setting.  (In the definition of \textit{cone chart}, the conditions that the cone chart $\phi$ be proper and that the space $Y$ be compact should be dropped, and a condition should be added which says that $\phi$ maps the collection of sets $\{ Y \times (r,\infty] : r > 0 \}$ to a neighborhood basis for the point $\phi(Y \times \{\infty\}) = \{p\}$.)   Also, as the referee notes, the result of \cite{4} used in the proof of the second sentence of the Theorem generalizes to the separable metric setting.  (See \cite{5}.)  However, because our proof of Lemma 2 relies totally on Theorem 2 of [1], the fate of Lemma 2 is unclear.  To the best of our knowledge, there is currently no analogue of Theorem 2 of [1] for separable metric spaces that are not locally compact.  This issue is discussed in Section 5 of [1].  See Example 2 and Question 3 on pages 52-53 of [1].  Also, a tangentially related example can be found in \cite{3}.  As a result, in the non-locally compact setting, one is unable to conclude that the space X is micro-homogeneous, thwarting the proofs of Lemma 2 and of the Theorem.

\section{Appendix: Alternate proof of Lemma 1}

\indent The alternative proof of Lemma 1 that we present here is inspired by the argument in \cite{7}.  (Also see \cite{14} and \cite{15}.)  First we generalize the notion of \textit{2-interlaced} cone charts.

\begin{definition} 
Let $\phi : Y \times (0,\infty] \rightarrow U$ and $\psi : Z \times (0,\infty] \rightarrow V$ be cone charts for open 
subsets $U$ and $V$ of a metric space $X$.  For $k \geq 2$, $\phi$ and $\psi$ are \emph{$k$-interlaced} if
\begin{itemize}
	\item $\phi(Y \times (2i-1,\infty]) \supset \psi(Z \times [2i,\infty])$ for $1 \leq i \leq k$ and
	\item $\psi(Z \times (2i,\infty]) \supset \phi(Z \times [2i+1,\infty])$ for $1 \leq i \leq k-1$.
\end{itemize}
\end{definition}

\indent The next lemma reveals that $2$-interlaced cone charts can be ``promoted'' to $k$-interlaced cone charts for all $k \geq 2$.

\begin{lemma}
If $k \geq 2$ and $\phi : Y \times (0,\infty] \rightarrow U$ and $\psi : Z \times (0,\infty] \rightarrow V$ are $k$-interlaced cone charts for open subsets $U$ and $V$ of a metric space $X$, then there is a cone chart $\phi' : Y \times (0,\infty] \rightarrow U$ for $U$ such that $\phi' = \phi$ on $Y \times (0,2k-1]$ and $\phi'$ and $\psi$ are $(k+1)$-interlaced.
\end{lemma}

\begin{proof}  
We explain the case $k = 2$.  The other cases are similar.

\indent There is an $r > 0$ such that $\phi(Y \times (1+r,\infty]) \supset \psi(Z \times [2,\infty])$, $\psi(Z \times (2+r,\infty]) \supset \phi(Z \times [3,\infty])$ and $\phi(Y \times (3+r,\infty]) \supset \psi(Z \times [4,\infty])$.

\indent Let $\lambda, \mu, \nu : (0,\infty] \rightarrow (0,\infty]$ be homeomorphisms with the following properties.
\begin{itemize}
	\item $\lambda = id$ on $(0,1] \cup \{\infty\}, \lambda(3) = 1+r$ and $\lambda(5) = 3$.
	\item $\mu = id$ on $(0,2] \cup [6,\infty]$ and $\mu(2+r) = 4$.
	\item $\nu = id$ on $(0,1], \nu|[1,1+r] = (\lambda|[1,3])^{-1}$ and $\nu = id$ on $[3+r,\infty])$.
\end{itemize}	
Next define homeomorphisms $\alpha, \beta, \gamma : X \rightarrow X$ as follows.
\begin{itemize}
	\item $\alpha(\phi(y,t)) = \phi(y,\lambda(t))$ for $(y,t) \in Y \times (0,\infty]$ and $\alpha = id$ on $X - \phi(Y \times [1,\infty])$.
	\item $\beta(\psi(z,t)) = \psi(z,\mu(t))$ for $(z,t) \in Z \times (0,\infty]$ and $\beta = id$ on $X - \psi(Z \times [2,\infty])$.
	\item $\gamma(\phi(y,t)) = \phi(y,\nu(t))$ for $(y,t) \in Y \times (0,\infty]$ and $\gamma = id$ on $X - \phi(Y \times [1,\infty])$.
\end{itemize}
Finally, define the map $\phi' : Y \times (0,\infty] \rightarrow X$ by $\phi' = \gamma \circ \beta \circ \alpha \circ \phi$.

\indent Since $\alpha, \beta$ and $\gamma$ are homeomorphisms of $X$ that map $U = \phi(Y \times (0,\infty])$ onto $U$, then $\phi'$ maps $Y \times (0,\infty]$ onto $U$.  Thus, $\phi'$ is a cone chart for $U$.

\indent Since $\alpha$ maps $\phi(Y \times (0,3])$ onto $\phi(Y \times (0,1+r])$, and $\phi(Y \times (0,1+r])$ is contained in the set $(X - U) \cup \psi(Z \times ((0,2))$ on which $\beta = id$, then $\gamma \circ \beta \circ \alpha = \gamma \circ \alpha$ on $\phi(Y \times (0,3])$.  Also $\nu \circ \lambda = id$ on $(0,3]$.  Hence, if $(y,t) \in Y \times (0,3]$, then $\phi'(y,t) =  \gamma \circ \alpha \circ \phi(y,t) = \phi(y,\nu \circ \lambda(t)) = \phi(y,t)$.  Thus, $\phi' = \phi$ on $Y \times (0,3]$.  Also $\phi'(Y \times (1,\infty]) = \phi(Y \times (1,\infty]), \phi'(Y \times [3,\infty]) = \phi(Y \times [3,\infty])$ and $\phi'(Y \times (3,\infty]) = \phi(Y \times (3,\infty])$.  Therefore, $\phi'$ and $\psi$ are $2$-interlaced.

\indent Since $\gamma = id$ on $\psi(Z \times [4,\infty])$, then $\psi(Z \times (4,\infty]) = \gamma \circ \psi(Z \times (4,\infty]) = \gamma \circ \beta \circ \psi(Z \times (2+r,\infty]) \supset \gamma \circ \beta \circ \phi(Y \times [3,\infty]) = \gamma \circ \beta \circ \alpha \circ \phi(Y \times [5,\infty]) = \phi'(Y \times [5,\infty])$.

\indent Since $\gamma \circ \beta = id$ on $\psi(Z \times [6,\infty])$, then $\psi(Z \times [6,\infty]) = \gamma \circ \beta \circ \psi(Z \times [6,\infty]) \subset \gamma \circ \beta \circ \phi(Y \times (3,\infty]) = \gamma \circ \beta \circ \alpha \circ \phi(Y \times (5,\infty]) = \phi'(Y \times (5,\infty])$.

\indent This proves $\phi'$ and $\phi$ are $3$-interlaced.
 
\indent This argument can be transformed into an argument for $k \geq 2$ by replacing the ``levels'' $1, 2, 3, 4, 5, 6$ by the ``levels'' $2k-3, 2k-2, 2k-1, 2k, 2k+1, 2k+2$. 
\end{proof}
 
 \indent The following lemma states that if two $2$-interlaced cone charts have vertices $p$ and $q$, then the cone chart with vertex $p$ can be perturbed to have vertex $q$ without moving it near its ``base''.  The proof of this lemma requires infinitely many applications of the previous lemma. This lemma is the key to the alternative proof of Lemma 1.
 
 \begin{lemma}
 If $\phi : Y \times (0,\infty] \rightarrow U$ and $\psi : Z \times (0,\infty] \rightarrow V$ are $2$-interlaced cone charts for open subsets $U$ and $V$ of a metric space $X$ with vertices $p$ and $q$, respectively, then there is a cone chart $\chi : Y \times (0,\infty] \rightarrow U$ for $U$ such that $\chi = \phi$ on $Y \times (0,3]$ and $\chi$ has vertex $q$.
 \end{lemma}
 
 \begin{proof}
 Starting from the given hypothesis, repeated use of Lemma 4 yields a sequence $\phi_i : Y \times (0,\infty] \rightarrow U, i \geq 2$ of cone charts for $U$ such that $\phi_2 = \phi$, $\phi_{i+1} = \phi_i$ on $Y \times (0,2i-1]$ and $\phi_i$ and $\psi$ are $i$-interlaced for $i \geq 2$.
 
 \indent Define $\chi : Y \times (0,\infty] \rightarrow U$ by $\chi = \phi_i$ on $Y \times (0,2i-1]$ for $i \geq 2$ and $\chi(Y \times \{\infty\}) = \{q\}$.  Then clearly $\chi = \phi$ on $Y \times (0,3]$.  We must prove that $\chi$ is a cone chart for $U$.    

\indent Since $\chi = \phi_i$ on $Y \times (0,2i-1)$ for $i \geq 2$, then it is easily seen that $\chi$ maps $Y \times (0,\infty)$ homeomorphically onto its image.  

\indent Since
\begin{align*}
		&U - \{q\} = U - \psi(Z \times \{\infty\}) \supset U - \psi(Z \times [2i,\infty]) \supset\\
		&U - \phi_i(Y \times [2i-1,\infty]) = \phi_i(Y \times (0,2i-1)) = \chi(Y \times (0,2i-1))
\intertext{for each $i \geq 2$, then $U - \{q\} \supset \bigcup_{i \geq 2}\chi(Y \times (0,2i-1)) = \chi(Y \times (0,\infty))$.  Since}
		&\chi(Y \times (0,\infty)) \supset \chi(Y \times (0,2i-1)) = \phi_i(Y \times (0,2i-1)) =\\
		&U - \phi_i(Y \times [2i-1,\infty]) \supset U - \psi(Z \times (2i-2,\infty])
\intertext{for each $i \geq 2$, then}
		&\chi(Y \times (0,\infty)) \supset \bigcup_{i \geq 2}(U - \psi(Z \times (2i-2,\infty])) =\\
		&U - \bigcap_{i \geq 2}(\psi(Z \times (2i-2,\infty])) = U - \psi(Z \times \{\infty\}) = U - \{q\}.
\end{align*}
This proves $U - \{q\} = \chi(Y \times (0,\infty))$.  Thus, $\chi$ maps $Y \times (0,\infty]$ onto $U$.

\indent Next we prove the continuity of $\chi$ at points of $Y \times \{\infty\}$.  To this end, let $W$ be a neighborhood of $\{q\} = \chi(Y \times \{\infty\}) = \psi(Z \times \{\infty\})$ in $X$.  Then $\psi^{-1}(W)$ is a neighborhood of 
$Z \times \{\infty\}$ in $Z \times (0,\infty]$.  Since $Z$ is compact, there is a $i \geq 1$ such that $\psi(Z \times (2i,\infty]) \subset W$.  Hence, for each $j > i$,
\begin{align*}
		&W \supset \psi(Z \times (2i,\infty]) \supset \phi_j(Y \times [2i+1,\infty]) \supset\\ 
		&\phi_j(Y \times [2i+1,2j-1]) = \chi(Y \times [2i+1,2j-1]).  
\end{align*}
Thus, $W \supset \bigcup_{j > i}\chi(Y \times [2i+1,2j-1]) = \chi(Y \times [2i+1,\infty))$.  Consequently, $\chi(Y \times (2i+1,\infty]) \subset W$, proving the continuity of $\chi$ at points of $Y \times \{\infty\}$.

\indent To complete the proof that $\chi$ is a cone chart for $U$, we must show $\chi : Y \times (0,\infty] \rightarrow U$ is a proper map.  To accomplish this, let $C$ be a compact subset of $U$.  Since $\chi$ is continuous and $\chi = \phi$ on $Y \times (0,3]$, then $\chi^{-1}(C)$ is a closed subset of $\phi^{-1}(C) \cup (Y \times [3,\infty])$. The latter set is compact because $\phi$ is proper.  Thus $\chi^{-1}(C)$ is compact. 
\end{proof}

\begin{proof}[Alternate proof of Lemma 1 using Lemma 5]
Suppose $\phi : Y \times (0,\infty] \rightarrow U$ and $\psi : Z \times (0,\infty] \rightarrow V$ are $2$-interlaced cone charts for open subsets $U$ and $V$ of a metric space $X$ with vertices $p$ and $q$, respectively.  Then Lemma 5 implies there is a cone chart $\chi : Y \times (0,\infty] \rightarrow U$ for $U$ such that $\chi = \phi$ on $Y \times (0,3]$ and $\chi$ has vertex $q$.  Define $h : X \rightarrow X$ by $h | U = \chi \circ \phi^{-1}$ and $h | X - U = id$.  Clearly, $h : X \rightarrow X$ is a homeomorphism and $h(p) = q$.  Since $h = id$ on $(X - U) \cup \phi(Y \times (0,3]) = X -  \phi(Y \times [3,\infty])$ and $\phi(Y \times [3,\infty]) \subset U \cap \psi(Z \times (2,\infty]) \subset U \cap V$, then h is supported on $U \cap V$. 
\end{proof}


\begin{thebibliography}{50}      

% after \begin{thebibliography}
\normalsize 


\bibitem[1]{1}F. D. Ancel, An alternative proof and applications of a theorem of E. G. Effros, \textit{Michigan Math. J.} \textbf{34} (1987), 39-55.

\bibitem[2]{2}D. P. Bellamy and J. M. Lysko, The generalized Schoenflies theorem for absolute suspensions, \textit{Colloquium Mathematicum} \textbf{103}(2005), 593-598.

\bibitem[3]{3}D. P. Bellamy and K. F. Porter, A homogeneous continuum that is non-Effros, \textit{Proc. Amer. Math. Soc.} \textbf{113}(1991), 241-246.

\bibitem[4]{4}R. Bennett, Countable dense homogeneous spaces, \textit{Fund. Math.} \textbf{74}(1972), 189-194.

\bibitem[5]{5}C. Bessaga and A. Pelczynski, The estimated extension theorem, homogeneous collections and skeletons, and their applications to the topological classification of linear metric spaces and convex sets, \textit{Fund. Math.} \textbf{69}(1970), 153-190.

\bibitem[6]{6}R. H. Bing and K. Borsuk, Some remarks concerning topologically homogeneous spaces, \textit{Annals of Math. (2)} \textbf{81} (1961), 812-814.

\bibitem[7]{7}M. Brown, The monotone union of open n-cells is an open n-cell, \textit{Proc. Amer. Math. Soc.}, \textbf{12} (1961), 100-111.

\bibitem[8]{8}J. Bryant, S. Ferry, W. Mio and S. Weinberger, Topology of homology manifolds, \textit{Annals of Math. (2)} \textbf{143} (1996), 435-467.

\bibitem[9]{9}J. Calcut, H. King and L. Siebenmann, Connected sum at infinity and Cantrell-Stallings hyperplane unknotting, \textit{Rocky Mountain J. Math.} \textbf{42} (2012), 1803-1862.

\bibitem[10]{10}J. Cantrell, Separation of the n-sphere by an (n-1)-sphere, \textit{Trans. Amer. Math. Soc.} \textbf{108} (1963), 185-194

\bibitem[11]{11}J. de Groot, On the topological characterization of manifolds, \textit{General Topology and its Relations to Modern Analysis and Algebra III (Proc. Third Prague Topological Symposium, 1971)}, Academia Prague, 1972, 155-158.

\bibitem[12]{12}E. G. Effros, Transformation groups and C*-algebras, \textit{Annals of Math. (2)} \textbf{81} (1965), 38-55.

\bibitem[13]{13}K. Kuperberg, W. Kuperberg and W. R. R. Transue, On the 2-homogeneity of Cartesian products, \textit{Fund. Math.} \textbf{110} (1980), 131-134.

\bibitem[14]{14}K. W. Kwun, Uniqueness of open cone neighborhoods, \textit{Proc. Amer. Math. Soc.} \textbf{15} (1964), 476-479.

\bibitem[15]{15}K. W. Kwun and F. Raymond, Mapping Cylinder Neighborhoods, \textit{Michigan Math. J.} \textbf{10} (1963), 353-357. 

\bibitem[16]{16}B. Mazur, On embeddings of spheres, \textit{Bull. Amer. Math. Soc.} \textbf{65} (1959), 59-65.

\bibitem[17]{17}W. J. R. Mitchell, Absolute suspensions and cones, \textit{Fund. Math.} \textbf{101} (1978), 241-244.

\bibitem[18]{18}S. Nadler, \textit{Continuum Theory: An Introduction}, Marcel Dekker, 1992.

\bibitem[19]{19}J. R. Stallings, On infinite processes leading to differentiability in the complement of a point, \textit{1965 Differential and Combinatorial Topology, A Symposium in honor of Marston Morse}, Princeton Univ. Press, 245-254. 

\bibitem[20]{20}A. Szymanski, Remarks on the absolute suspension, \textit{Fund. Math.} \textbf{86} (1974), 157-161.



\end{thebibliography}
\end{document}